\documentclass[11pt,dvips,twoside,letterpaper]{article} 
\usepackage{times,fancyhdr}
\usepackage{graphicx}
\usepackage{geometry}
\usepackage{titlesec}

\titlelabel{\thetitle.\quad}                                

\usepackage[labelsep=period]{caption}                       

\makeatletter 
\def\cleardoublepage{\clearpage\if@twoside \ifodd\c@page\else%
    \hbox{}%
    \thispagestyle{empty}
    \newpage%
    \if@twocolumn\hbox{}\newpage\fi\fi\fi} 
\makeatother

\usepackage{amssymb,amsmath,amsthm,amsfonts,enumerate,verbatim,color}
\usepackage{appendix}

\newtheorem*{Ack}{Acknowledgment}

\newtheorem{theorem}{Theorem}
\newtheorem*{theorem4*}{Theorem 4 (restated)}
\newtheorem*{theorem1*}{Theorem 1 (restated)}

\newcommand{\defeq}{:=}
\newcommand{\two}{2}
\newcommand{\graph}{G}
\newcommand{\legs}{e}
\newcommand{\map}{\varphi}
\newcommand{\CC}{\mathcal{C}}
\newcommand{\card}{\mbox{card}}

\newcommand{\lp}{k}
\newcommand{\vertex}{n}
\newcommand{\edge}{m}
\newcommand{\iso}{f}

\def\b{\bigg}

\def\({\b(}
\def\[{\b[}
\def\){\b)}
\def\]{\b]}


\begin{document}
\title{\bfseries\scshape{Several Graph Sequences as Solutions\\ of a Double 
Recurrence\thanks{The second author performed this work during a postdoctorate 
at the University of Lisbon and later at Western Illinois University. 
The third author performed this work within the activities
of the Centro de Estruturas Lineares e Combinat\'orias (University of Lisbon,
Portugal) and was supported by the fellowship SFRH/BPD/48223/2008
provided
by the Portuguese Foundation for Science and Technology (FCT).}}}
\author{\bfseries\itshape Christian Brouder$^{1,}$\thanks{E-mail address: 
Christian.Brouder@impmc.upmc.fr}~,  William J. Keith$^{2,}$\thanks{E-mail 
address: wjkeith@mtu.edu} ~and \^Angela Mestre$^{3,}$\thanks{E-mail address: 
amestre@fc.ul.pt}\\
$^{1}$Institut de Min\'{e}ralogie, de Physique des Mat\'{e}riaux
 et de Cosmochimie,\\ 
Universit\'{e} Pierre et Marie Curie-Paris 6, 4 place Jussieu,
75005 Paris, France\\
$^{2}$Michigan Technological University, US\\
$^{3}$Center
for Functional Analysis, Linear Structures and Applications (CEAFEL)\\
Group for Linear, Algebraic and Combinatorial Structures\\
Departamento de Matem\'atica\\
Faculdade de Ci\^encias da
Universidade de Lisboa,
Portugal}
\date{Received: January 13, 2014; Accepted: October 27, 2014}
\maketitle
\thispagestyle{empty}
\setcounter{page}{37}
\thispagestyle{fancy}
\fancyhead{}
\fancyhead[L]{Journal of Combinatorics and Number Theory  \\ 
Volume 6, Number 2, pp. {\thepage--\pageref{lastpage-01}}} 
\fancyhead[R]{ISSN 1942-5600  \\
\copyright{} 2015 Nova Science Publishers, Inc.}
\fancyfoot{}
\renewcommand{\headrulewidth}{0pt}

\begin{abstract} We describe the combinatorics that arise in summing a double 
recursion formula for the enumeration of connected Feynman 
graphs in quantum field theory.  In one index 
the problem is more tractable and 
yields concise formulas which are combinatorially interesting on their own.  In 
the other index, one of these sums is Sloane's sequence {A001865}.
\end{abstract}




\section{Introduction}\label{sec:intro}
\setcounter{equation}{0} \setcounter{Thm}{0} \setcounter{Lem}{0}
\setcounter{Cor}{0}\setcounter{Conj}{0}

The present paper was originally motivated by the desire to produce a closed 
form for a recursion for enumerating connected Feynman graphs in quantum 
field theory. These are combinatorial species whose properties often resemble 
those of multigraphs with multiple edges and loops allowed. 
The attempt to reduce the recurrence to 
closed 
formulae turns out to involve interesting combinatorics with connections to 
several standard graph sequences, the description of which is the purpose of this paper.

Our formula is related to the formula of Proposition 15 of \cite{MeOe:loop}. 
The latter generates tensors representing connected Feynman graphs parametrized 
by their vertex and cyclomatic numbers. 
Such formula induces a recurrence
for the 
sum of the inverses of the orders of the groups of automorphisms of all the 
pairwise non-isomorphic connected Feynman graphs on $n$ vertices and 
cyclomatic 
number $k$, denoted by $I(n,k)$. For all $n\geq1$ and $k\geq0$ this reads as 
follows:

$$I(n,k) = \frac{1}{2(n+k-1)} \times \\
\left(n^2I(n,k-1)+\sum_{i=1}^{n-1}\sum_{j=0}^k 
i(n-i)I(i,j)I(n-i,k-j)\right)\,.$$

Here we solve the recurrence above for any constant $n$ or for small $k$.

In constant $n$, the answer is more complete.  The usual standard in 
combinatorics for a well-understood sequence is to produce a concise generating 
function.  We find several forms of the generating function $$R_n(x) = 
\sum_{k=0}^{\infty} I(n,k) x^k,$$ \noindent including the two-variable 
generating function

\pagestyle{fancy}
\fancyhead{}
\fancyhead[EC]{Christian Brouder, William J. Keith and \^Angela Mestre}
\fancyhead[EL,OR]{\thepage}
\fancyhead[OC]{Several Graph Sequences as Solutions of a Double Recurrence}
\fancyfoot{}
\renewcommand\headrulewidth{0.5pt} 

\begin{theorem}\label{RnFunction} $\sum_{n=1}^{\infty} R_n(x) (sx)^n = x \log 
\left( \sum_{n=0}^{\infty} \frac{s^n}{n!} e^{xn^2/2} \right).$\end{theorem}

Define $J(n,k)\defeq 2^k (n+k-1)!  I(n,k)$ to normalize the sequences $I(n,k)$. 
 We find that the sequences of $J$ in constant $n$ possess concise rational 
generating functions whose coefficients are simple formulas.  With the last sum 
running over all compositions of $n$, i.e. sequences of positive integers that 
sum to $n$, the general form is:

\begin{theorem}$$\sum_{k \geq 0} J(n,k) t^k = -\frac{1}{(2t)^{n-1}} 
\sum_{m=1}^n \frac{(-1)^m}{m}  \sum_{(n_1,\dots,n_m)\vdash n} 
\frac{1}{n_1!\dots 
n_m!} \frac{1}{1-(n_1^2+\dots+n_m^2)t}.$$
\end{theorem}

There are $2^{n-1}$ compositions of $n$, but the number of distinct sums of 
squares of parts among these compositions is much smaller, meaning that the 
number of factors in each generating function is much smaller than might have 
been expected from the above theorem.  Hence a more efficient method of 
producing a generating function and explicit formula for a given $n$ might be 
desired.  There is a fairly simple routine for calculating explicit formulas 
and generating functions, which we give in the last section. Values for $n=1$ 
through 4 are given in the table in the theorem below:

\begin{theorem}For $k \geq 0$, the formulas and generating functions for 
$J(n,k)$, $1 \leq n \leq 4$, are as given in the following table.

{\footnotesize \begin{center}\begin{tabular}{|c|c|c|}
\hline $J(n,k)$ & Formula & Generating function $\sum_{k=0}^{\infty} J(n,k) 
x^k$\\
\hline $J(1,k)$ & $1$ &  $\frac{1}{1-x}$ \\
\hline $J(2,k)$ & $4^k - 2^{k-1}$ & $\frac{1/2}{(1-2x)(1-4x)}$ \\
\hline $J(3,k)$ & $\left( \frac{-25}{8} \right) 5^k + \left( \frac{3}{4} 
\right) 3^k + \left( \frac{27}{8} \right) 9^k$ & $\frac{1}{(1-3x)(1-5x)(1-9x)}$ 
\\
\hline $J(4,k)$ & $\left( \frac{64}{3} \right) 16^k + \left( \frac{27}{2} 
\right) 6^k - \left( \frac{250}{12} \right) 10^k - (2) 4^k - (8) 8^k$ & 
$\frac{4-34x}{(1-4x)(1-6x)(1-8x)(1-10x)(1-16x)}$ \\
\hline 
\end{tabular}\end{center}}
\end{theorem}
\noindent Note that $J(3,k)$ is the sequence A017897, while $J(4,k)/2$ 
is sequence A221959 in \cite{OEIS}.

For $k=0$ or $k=1$ the calculations are more difficult.  We have

\begin{theorem}For all $n \geq 1$, $I(n,0) = \frac{n^{n-2}}{n!}$ and $I(n,1) = 
\frac{1}{2} \sum_{\mu=1}^n \frac{n^{n-\mu-1}}{(n-\mu)!}$.\end{theorem}

As expected, for $k=0$ the recurrence is solved by the sum of the inverses of 
the orders of the groups of automorphisms of all the trees (up to isomorphism) 
on $n$ vertices. This clearly equals $n^{n-2}/n!$ for an isomorphism of trees 
on $n$ vertices is defined by the action of the symmetric group on $n$ symbols, 
while the number of such trees is given by Cayley's formula $n^{n-2}$ 
\cite{Polya}. For $k=1$ if we multiply the $n$th term of $I(n,1)$ by $2n!$ we 
obtain the $n$th term of Sloane's sequence  A001865 in \cite{OEIS}. 
This is thus proven to be a recurrence sequence. Note that Flajolet, Knuth, and Pittel  
showed that Sloane's sequence is related to the enumeration of connected 
multigraphs with only one cycle \cite{FKB}.

This paper is organized as follows. In 
Section \ref{sec:enumerator} we prove Theorem 1 by solving the differential 
recursive equation. We also give a general expression for $J(n,k)$. In Section 
\ref{sec:hands} we consider $R_n$ from another angle, and re-prove some of the 
expansions with a more hands-on approach that displays different underlying 
combinatorics.  In Section \ref{sec:sloane} we prove Theorem 4.  When $k=0$, we 
use Dziobek's recurrence \cite{dziobek} for Cayley's formula  to show that 
$J(n,0)=n^{n-3}$. When $k=1$, we use an identity related to Abel's theorem to 
prove that  $J(n,1)= n!\sum_{\mu=1}^{n} n^{n-\mu-1}/(n-\mu)!$, which is 
sequence A001865 in \cite{OEIS}.  Finally, in the appendix for 
completeness, we give the mathematical definitions of Feynman graphs and their 
automorphisms in the context of this paper.

\section{Formulas in Constant $n$}\label{sec:enumerator}

We use the formula of Proposition 15 of \cite{MeOe:loop} to give a recursive 
definition for the numbers  $I(n,k)$.  
With 
boundary condition 
$$I(n,k)=0 \, \text{ for } \, k<0 \, \text{ and/or } \, n<1 \, \text{, and } 
I(1,0)=1$$

\noindent then for all $n\geq1$ and $k\geq0$ we have
\begin{multline}\label{eq:Izero}
I(n,k) = \frac{1}{2(n+k-1)} \times \\
\left(n^2I(n,k-1)+\sum_{i=1}^{n-1}\sum_{j=0}^k i(n-i)I(i,j)I(n-i,k-j)\right)\,.
\end{multline}

For convenience of 
calculation, we define 
\begin{equation}\label{eq:J}
J(n,k) \defeq 2^k (n+k-1)! I(n,k)\,.
\end{equation}

\noindent Then, for all $n\geq1$ and $k\geq0$  the recursion for $J(n,k)$ is
\begin{multline}\label{eq:Jzero}
J(n,k) = n^2 J(n,k-1) + \\ \frac{1}{2} \sum_{i=1}^{n-1} \sum_{j=0}^k 
\binom{n+k-2}{i+j-1} i(n-i)J(i,j)J(n-i,k-j)\,.
\end{multline}

The same boundary conditions hold.

\subsection{Generating Functions}

Define the generating function for the sequence of $I(n,k)$ with $n$ constant:
\begin{eqnarray}
R_n(x) &=& \sum_{k=0}^\infty I(n,k)x^k. 
\label{Rn}
\end{eqnarray}
In particular,
\begin{eqnarray}
R_1(x)  = \sum_{k=0}^\infty I(1,k)x^k 
  = \sum_{k=0}^\infty \frac{J(1,k)}{2^k k!} x^k = e^{x/2},
\end{eqnarray}
\noindent because $J(1,k)=1$ for all $k$.

We now wish to prove 

\begin{theorem1*} $$\sum_{n=1}^{\infty} R_n(x) (sx)^n = x \log \left( 
\sum_{n=0}^{\infty} \frac{s^n}{n!} e^{xn^2/2} \right).$$
\end{theorem1*}

\begin{proof} If we multiply each $I(n,k)$ by $x^k$ and sum, the recursive 
equation (\ref{eq:Izero}) for $I(n,k)$ becomes a differential recursive 
equation for $R_n$:
\begin{eqnarray*}
2(n-1) R_n(x) + 2x R'_n(x) &=&
  n^2 x R_n(x) + \sum_{i=1}^{n-1} i(n-i) R_i(x) R_{n-i}(x), 
\end{eqnarray*}
\noindent with boundary conditions $R_n(0)=I(n,0)=n^{n-2}/n!$, $R_n(x) = 0$ for 
$n<1$.

To simplify, apply the change of variable 
\begin{eqnarray*}
R_n(x) &=& P_n(e^x) \frac{e^{nx/2}}{n! x^{n-1}}.
\label{Pn}
\end{eqnarray*}

The inverse relation
\begin{eqnarray*}
P_n(z) &=& n! (\log z)^{n-1} z^{-n/2} R_n(\log z)
\end{eqnarray*}
gives the boundary condition $P_n(1)=\delta_{n,1}$.

The differential recursive equation satisfied by $P_n(z)$ is
\begin{eqnarray*}
2 z P'_n(z) &=& n(n-1) P_n(z) + \sum_{i=1}^{n-1} i(n-i)\binom{n}{i}
 P_i(z) P_{n-i}(z).
\end{eqnarray*}

Observe that $P_1(z) = 1$, and if $P_i(z)$ is polynomial for all $1 \leq i < 
n$, then $P_n(z)$ satisfying this recurrence is polynomial as well.  The first 
polynomials are
\begin{eqnarray*}
P_1(z) &=& 1,\\
P_2(z) &=& -1+z,\\
P_3(z) &=& 2 - 3 z + z^3,\\
P_4(z) &=& -6 + 12 z - 3 z^2 - 4 z^3 + z^6,\\
P_5(z) &=& 24 - 60 z + 30 z^2 + 20 z^3 - 10 z^4 - 5 z^6 + z^{10},\\
P_6(z) &=& -120 + 360 z - 270 z^2 - 90 z^3 + 120 z^4 + 20 z^6 - 15 z^7 
  - 6 z^{10} + z^{15},\\
P_7(z) &=& 720 - 2520 z + 2520 z^2 + 210 z^3 - 1260 z^4 + 210 z^5 -
70 z^6 + 210 z^7 - 35 z^9 \\ && + 42 z^{10} - 21 z^{11} - 7 z^{15} + z^{21}.
\end{eqnarray*}
The coefficients of the polynomials $P_n(z)$ form the sequence A221960 in 
\cite{OEIS}.
The equation for the generating function
\begin{eqnarray*}
f(y,z) &=& \sum_{n=1}^\infty P_n(z) \frac{y^n}{n!} 
\end{eqnarray*}

\noindent is
\begin{eqnarray*}
2z \frac{\partial f}{\partial z} &=& y^2 
  \frac{\partial^2 f}{\partial y^2}
  + \Big(y\frac{\partial f}{\partial y}\Big)^2,
\end{eqnarray*}

\noindent with the boundary condition
\begin{eqnarray*}
f(y,1) &=& \sum_{n=1}^\infty P_n(1) \frac{y^n}{n!} =y.
\end{eqnarray*}

The non-linear term is eliminated by the change of variable $f=\log g$. The 
equation for $g$ is therefore linear:
\begin{eqnarray*}
2z \frac{\partial g}{\partial z} &=& y^2 
  \frac{\partial^2 g}{\partial y^2},
\end{eqnarray*}

\noindent with the boundary condition $g(y,1)=e^{f(y,1)}=e^y$. A family of 
solutions of this equation is given by
\begin{eqnarray*}
\sum_{n=0}^\infty g_n \frac{y^n z^{n(n-1)/2}}{n!}\,.
\end{eqnarray*}

\noindent Moreover, the boundary condition $g(y,1)=e^y$ implies $g_n=1$ for all 
$n$:
\begin{eqnarray*}
g(y,z) &=& \sum_{n=0}^\infty \frac{y^n z^{n(n-1)/2}}{n!}.
\end{eqnarray*}

\noindent Finally,
\begin{eqnarray*}
f(y,z) &=& \sum_{n=1}^\infty P_n(z) \frac{y^n}{n!} 
= \log\Big(\sum_{n=0}^\infty \frac{y^n z^{n(n-1)/2}}{n!}\Big).
\end{eqnarray*}

Note that if we set $z = (1+t)$, we obtain
\begin{eqnarray*}
f(y,t) = \log\Big(\sum_{n=0}^\infty \frac{y^n (1+t)^{n(n-1)/2}}{n!}\Big) = 
\sum_{n=0}^{\infty} C_n(t) \frac{y^n}{n!},
\end{eqnarray*}

\noindent which is the well-known exponential generating function for  $C_n(t) 
= \sum_G t^{e(G)}$, where the sum is over all connected graphs $G$ on the set 
$\{1,2,\dots,n\}$ and $e(G)$ is the number of edges of $G$.  (See \cite{Gessel} 
for this formula and related results.)  Thus, $P_n(z) = C_n(z-1)$.

Substituting $e^x$ for $z$, we obtain a generating function for $R_n$:
\begin{eqnarray*}
f(y,e^x) &=& \sum_{n=1}^\infty P_n(e^x) \frac{y^n}{n!} 
=\frac{1}{x} \sum_{n=1}^\infty R_n(x) (xye^{-x/2})^n.
\end{eqnarray*}

The generating function for $R_n(x)$ becomes
\begin{eqnarray*}
\sum_{n=1}^\infty R_n(x) s^n
&=& x f(se^{x/2}/x,e^x)
=x \log\Big(\sum_{n=0}^\infty \frac{(s/x)^n}{n!} e^{xn^2/2}\Big).
\end{eqnarray*}

However, this expression is highly singular in $x$ and does not
provide a generating function for $R_n$. The alternative expression 
\begin{eqnarray*}
x h(x,s) &=& \sum_{n=1}^\infty R_n(x) (sx)^n
=x f(se^{x/2},e^x)
=x \log\Big(\sum_{n=0}^\infty \frac{s^n}{n!} e^{xn^2/2}\Big),
\end{eqnarray*}
is not singular and generates $R_n$, in the sense
that $R_n(x)$ is $x^{-n}$ times the coefficient of $s^n$ in the
expansion of $xh(x,s)$. \end{proof}

The extraction described yields
\begin{eqnarray*}
R_n(x) &=& -x^{1-n} \sum_{m=1}^n \frac{(-1)^m}{m}
  \sum_{n_1+\dots+n_m=n} \frac{e^{x(n_1^2+\dots+n_m^2)/2}}
     {n_1!\dots n_m!},
\end{eqnarray*}
where all $n_i\ge1$.

\subsection{Explicit Forms for $J(n,k)$}

We now wish to show Theorem 3 by isolating explicit generating functions for 
$J(n,k)$ individually.
We can proceed as follows:
\begin{eqnarray*}
x h(x,s) &=& \sum_{n=1}^\infty \sum_{k=0}^\infty
  I(n,k) s^n x^{n+k}.
\end{eqnarray*}

Now replace $x \rightarrow 2x$, $s \rightarrow s/2$ to obtain

\begin{eqnarray*}
2x h(2x,s/2) &=& \sum_{n=1}^\infty \sum_{k=0}^\infty
  2^k I(n,k) s^n x^{n+k}.
\end{eqnarray*}

We define $\beta(s,t) = 2 \int_{0}^\infty d x e^{-x} h(2xt,s/2)$ to get
\begin{eqnarray*}
\beta(s,-t) &=& 2 \int_{0}^\infty d x 
  e^{-x} h(-2xt,s/2) =
\sum_{n=1}^\infty \sum_{k=0}^\infty
\int_{0}^\infty d x 
  e^{-x} 2^k I(n,k) s^n (-xt)^{n+k-1} 
\\&=&
\sum_{n=1}^\infty \sum_{k=0}^\infty
  2^k (n+k-1)! I(n,k)s^n (-t)^{n+k-1},
\end{eqnarray*}

\noindent where we used $\int_0^\infty e^{-x} x^p dx=p!$. Therefore,
\begin{eqnarray*}
\beta(s,t) &=& 
\sum_{n=1}^\infty \sum_{k=0}^\infty
  J(n,k)s^n t^{n+k-1}=
\sum_{n=1}^\infty Z_n(t) s^n t^{n-1},
\end{eqnarray*}

\noindent is a generating function for $J(n,k)$ and this equation defines 
$Z_n(t)=\sum_{k=0}^\infty J(n,k) t^k$. We calculate
\begin{eqnarray*}
\beta(s,-t) &=& 2 \int_{0}^\infty d x 
  e^{-x} \log\Big(\sum_{n=0}^\infty \frac{(s/2)^n}{n!} e^{-n^2 xt}\big)
=2 \int_0^1 d\lambda 
  \log\Big(\sum_{n=0}^\infty \frac{(s/2)^n}{n!} \lambda^{n^2t}\big),
\end{eqnarray*}

\noindent where we have put $\lambda=e^{-x}$.  This gives 
\begin{eqnarray*}
\beta(s,-t) &=& 
 -2 \sum_{m=1}^n \frac{(-1)^m}{m}
  \sum_{n_1,\dots,n_m} \frac{(s/2)^{n_1+\dots+n_m}}
     {n_1!\dots n_m!}
  \int_0^1 d\lambda
  \lambda^{t(n_1^2+\dots+n_m^2)}.
\end{eqnarray*}

Thus,
\begin{eqnarray*}
\beta(s,t) &=& 
 -2 \sum_{m=1}^n \frac{(-1)^m}{m}
  \sum_{n_1,\dots,n_m} \frac{(s/2)^{n_1+\dots+n_m}}
     {n_1!\dots n_m!}
  \frac{1}{1-(n_1^2+\dots+n_m^2)t}.
\end{eqnarray*}

The polynomials $Z_n(t)$ can be calculated by
summing over the compositions of $n$, proving Theorem 3:
\begin{eqnarray*}
Z_n(t) &=& 
 -\frac{1}{(2t)^{n-1}} \sum_{m=1}^n \frac{(-1)^m}{m}
  \sum_{n_1+\dots+n_m=n} \frac{1}{n_1!\dots n_m!}
  \frac{1}{1-(n_1^2+\dots+n_m^2)t}.
\end{eqnarray*}

\hfill $\Box$

The number of distinct sums of squares of the parts of partitions or 
compositions of $n$ is the OEIS sequence A069999; it grows like 
$n^2/2$~\cite{savit-00}.  (We remark that the first case where a sum appears 
twice is $n=6$.)  The first $Z_n(t)=\sum_{k\geq0} J(n,k)t^k$ are:
{\small \begin{eqnarray*}
Z_1(t) &=& \frac{1}{1-t},\\
Z_2(t) &=& \frac{1/2}{(1-2t)(1-4t)},\\
Z_3(t) &=& \frac{1}{(1-3t)(1-5t)(1-9t)},\\
Z_4(t) &=& \frac{4-34t}{(1-4t)(1-6t)(1-8t)(1-10t)(1-16t)},\\
Z_5(t) &=& \frac{25-606t+3557t^2}{(1-5t)(1-7t)(1-9t)(1-11t)(1-13t)
  (1-17t)(1-25t)},\\
Z_6(t) &=& \frac{24(9-451t+7292t^2-37860t^3)}
  {(1-6t)(1-8t)(1-10t)(1-12t)(1-14t)(1-18t)(1-20t)(1-26t)(1-36t)}.
\end{eqnarray*}}

\section{Formulas in Constant $n$: The Hands-on Approach}\label{sec:hands}

In this section we are interested in re-deriving the $Z_n(t)$ more directly, 
and finding the closed formulas to prove Theorem 4.  Using strictly discrete 
mathematics instead of the generating function, we seek additional insight into 
the combinatorial structure of the sequences $J(n,k)$.

The case $J(1,k)=1$ for $k>0$ is trivial from the recursion.

For $n=2$, the recursion (\ref{eq:Jzero}) becomes
\begin{eqnarray*}
J(2,k) & = & 4J(2,k-1) + \frac{1}{2}\sum_{j=0}^k \binom{k}{j} J(1,j) J(1,k-j) \\
 & = & 4J(2,k-1) + \frac{1}{2} \sum_{j=0}^{k} \binom{k}{j} = 4J(2,k-1) + 
2^{k-1}.
\end{eqnarray*}

We now note that $4(4^{k-1}-2^{k-2}) + 2^{k-1} = 4^k - 2^k + 2^{k-1} = 4^k - 
2^{k-1}$ and since the statement holds for $k=0$, the claim on the values 
follows inductively.

That the value $4^k - 2^{k-1}$ is the coefficient of $x^k$ in 
$\frac{1}{2(1-2x)(1-4x)}$ can be seen from expansion: the coefficient is
\begin{eqnarray*}
\frac{1}{2} \left(4^k + 4^{k-1}2 + 4^{k-2}2^2 + \dots + 2^k \right) & = & 
\frac{1}{2}4^k\left(1+\frac{1}{2}+\frac{1}{4} + \dots + \frac{1}{2^k}\right) \\
 & = & \frac{1}{2}4^k \left( 2 - \frac{1}{2^k} \right) = 4^k - 2^{k-1}.
\end{eqnarray*}

For $n=3$, the recursion (\ref{eq:Jzero}) gives
\begin{multline}\label{JThreeK}
J(3,k)  =  9J(3,k-1) + \\
   \frac{1}{2}\cdot 2 \sum_{j=0}^k \left[ \binom{k+1}{j} J(1,j)J(2,k-j) + 
\binom{k+1}{j+1} J(2,j) J(1,k-j)  \right] \\
  =  9J(3,k-1) + \sum_{j=0}^k \left[ \binom{k+1}{j} \left(4^{k-j}-2^{k-j-1} 
\right) + \binom{k+1}{j+1} \left( 4^j - 2^{j-1}\right) \right] \\
   =  9J(3,k-1) + 2 \sum_{j=0}^k \binom{k+1}{j+1} \left( 4^j - 2^{j-1} \right) 
\\
    =  9J(3,k-1) + \frac{1}{2} \sum_{i=0}^{k+1} \binom{k+1}{i} 4^i - 
\frac{1}{2} \sum_{i=0}^{k+1} \binom{k+1}{i} 2^i \\
    =  9J(3,k-1) + \frac{1}{2} (1+4)^{k+1} - \frac{1}{2} (1+2)^{k+1} \\
    =  9J(3,k-1) + \frac{1}{2} 5^{k+1} - \frac{1}{2} 3^{k+1}.
\end{multline}

Let us pause to note the process here.  We added a term to the sum, namely the 
$i=0$ term, which is the $j=-1$ case of $4^j - 2^{j-1}$.  However, this term is 
0.  We then complete the binomial sum, changing the 4 to a 5, and the 2 to a 3.

Iterating the recursion with the form thus obtained, we obtain the closed 
formula for $J(3,k)$:
\begin{eqnarray*}
J(3,k) & = & 9J(3,k-1) + \frac{1}{2}5^{k+1} - \frac{1}{2} 3^{k+1} \\
 &=& 9^2 J(3,k-2) + 9 \cdot \frac{1}{2} \left(5^k - 3^k\right) + \frac{1}{2} 
\left(5^{k+1} - 3^{k+1}\right) \\
 &=& \dots = 9^k J(3,0) + \frac{1}{2}\left( 9^{k-1} \left( 5^2 - 3^2 \right) + 
\dots + \left( 5^{k+1} - 3^{k+1} \right) \right) \\
 &=& \frac{1}{2} 9^{k+1} \cdot \left( \frac{1}{9} \left(5^1 - 3^1 \right) + 
\dots + 9^{-(k+1)} \left( 5^{k+1} - 3^{k+1} \right) \right) \\
  &=& \frac{1}{2} 9^{k+1} \left[ \sum_{i=1}^{k+1} \left( \frac{5}{9} \right)^i 
- \left( \frac{3}{9} \right)^i\right] \\
  &=& \left( \frac{81}{24} \right) 9^k + \left( \frac{-75}{24} \right) 5^k + 
\left( \frac{18}{24} \right) 3^k .
\end{eqnarray*}

To show that the generating function claim is true, we factor the last term of 
equation (\ref{JThreeK}):
\begin{eqnarray*}
J(3,k) &=& 9J(3,k-1) + \frac{1}{2}5^{k+1} - \frac{1}{2}3^{k+1} \\
&=& 9 J(3,k-1) + \frac{1}{2}(5-3) \sum_{i=0}^k 5^{k-i}3^i \\
&=& 9J(3,k-1) + \sum_{i=0}^k 5^{k-i}3^i .
\end{eqnarray*}

The same recurrence and initial condition is satisfied by the coefficient of 
$x^k$ in $\frac{1}{(1-3x)(1-5x)(1-9x)}$, as we can see by noting that this 
coefficient is the complete homogeneous symmetric polynomial in 3, 5, and 9: 
\begin{multline*}
\left[ x^k \right] \frac{1}{(1-3x)(1-5x)(1-9x)} = \sum_{{a+b+c=k} \atop {a,b,c 
\in \mathbb{N}}} 9^a5^b3^c \\ = 9 \sum_{{a+b+c=k-1} \atop {a,b,c \in 
\mathbb{N}}} 9^a 5^b 3^c + \sum_{i=0}^k 5^{k-i}3^i \, \text{.}
\end{multline*}

\noindent where $\left[ x^k \right] f(x)$ denotes the coefficient of $x^k$ in 
$f(x)$.

The $n=4$ case is proved similarly.  Begin by expanding the defining recursion 
for $J(n,k)$ using the formulas for $J(1,k)$ through $J(3,k)$.  Gather 
symmetric expansions with the same binomial coefficients.
\begin{multline*}
J(4,k)= 16 J(4,k-1)+ \\ 
\sum_{j=2}^{k+2} \binom{k+2}{j} \left( \left(\frac{81}{8} \right) 9^{j-2} + 
\left(\frac{-75}{8} \right) 5^{j-2} + \left( \frac{18}{8} \right) 3^{j-2} 
\right) + \\
2 \sum_{j=1}^{k+1} \binom{k+2}{j} \left( 4^{j-1} - 2^{j-2} \right) \left( 
4^{k-j+1} -2^{k-j} \right) .
\end{multline*}

Now complete the sums so that the bounds are $\sum_{j=0}^{k+2}$.  The crucial 
observations at this point are again that $4^i-2^{i-1}=0$ for $i=-1$, and 
likewise $\left(\frac{81}{8} \right) 9^i + \left(\frac{-75}{8} \right) 5^i + 
\left( \frac{18}{8} \right) 3^i=0$ for $i\in\{-1,-2\}$, and so the terms to be 
added are all 0.

Using the binomial theorem to sum, we find 
\begin{multline*}
J(4,k) = 16J(4,k-1) + \left( \frac{25}{2} \right) 10^k + \left( \frac{-45}{2} 
\right) 6^k + \left( 6 \right) 4^k + \left( 8 \right) 8^k .
\end{multline*}

\noindent At this point the desired generating function can be shown to have 
coefficients given by the same recursion.  By iterating the recursion and 
summing the truncated geometric series, we obtain the formula given in Theorem 
4. 

\hfill $\Box$

The values for the formula for $J(4,k)$ are indeed zero at $k\in\{-1,-2,-3\}$, 
and so a similar neat expansion holds for $J(5,k)$.  This property holds 
forever (to argue this, observe the factor $1/(2t)^{n-1}$ in the expansion of 
$Z_n(t)$, and note that no automorphism has negative order, so that the 
$Z_n(t)$ have no negative order terms; this implies that the sums vanish for 
$-n<k<0$).  The process of completing the binomial, summing with the binomial 
formula, and then iterating the recursion and summing the geometric series will 
thus provide explicit formulas easily.  

From the expansion of $Z_n(t)$ at the end of the previous section it is clear 
that the sequences $J(n,\cdot)$ will all be tidy exponential sums with number 
of terms given by A069999, and exponential bases themselves given by 
the sums 
of squares being counted by that sequence.  The coefficients, then, constitute 
the hard part of the formulae to know in advance.

\section{Formulas in Constant $k$}\label{sec:sloane}

In this section we prove Theorem 4, showing that the sequences that solve our 
recurrence for $k=0,1$ are respectively A007830 and A001865
in \cite{OEIS}.

\begin{theorem4*} For all $n\ge 1$:
\begin{enumerate}[(a)]
\item
$J(n,0) = n^{n-3}\,,\quad\mbox{or equivalently,}\quad  I(n,0) = 
\frac{n^{n-2}}{n!}\,.$
\item 
$J(n,1)=n!\sum_{\mu=1}^n \frac{n^{n-\mu-1}}{(n-\mu)!}\,,\quad\mbox{or 
equivalently,}\quad  I(n,1)=\frac{1}{2} \sum_{\mu=1}^n 
\frac{n^{\mu-2}}{(\mu-1)!}\,.$
\end{enumerate}
\end{theorem4*}

\begin{proof} To prove (a), note that the recurrence for $J(n,0)$ reads as 
follows:
\begin{eqnarray}\nonumber
J(1,0)&=&1\,,\\\label{eq:Dzione}
J(n,0) & =& \frac{1}{2n(n-1)} \sum_{i=1}^{n-1} \binom{n}{i} 
i^2(n-i)^2J(i,0)J(n-i,0)\,.
\end{eqnarray}

For $T(n)\defeq nJ(n,0)$ we obtain Dziobek's recurrence \cite{dziobek} for 
Cayley's formula:
\begin{eqnarray}\nonumber
T(1)&=&1\,,\\\label{eq:DziobekT}
T(n) & =& \frac{1}{2(n-1)} \sum_{i=1}^{n-1} \binom{n}{i} i(n-i)T(i)T(n-i)\,.
\end{eqnarray}

This is thus solved by $T(n)=n^{n-2}$ which is the sequence A000272 in 
\cite{OEIS}. Therefore, we have $J(n,0)=n^{n-3}$. 

\bigskip

To prove (b)  
we first recall the following Abel-type identity \cite[p. 93]{Riordan}:

$$n(n+y)^{n-1} = \sum_{i=1}^n \binom{n}{i} i(-x+1)(-x+i)^{i-2}(x+y+n-i)^{n-i} 
.$$

For $x=0$ and $y=j$ the above formula specializes to 
\begin{equation}\label{AbelFormula} n(n+j)^{n-1} = \sum_{i=1}^n \binom{n}{i} 
i^{i-1} (n+j-i)^{n-i} .\end{equation}

We now proceed by induction on $n$.  The result clearly holds for $n=1$. We 
assume the result to hold for all $t$ in $0 < t < n$.  Then
\begin{eqnarray*} J(n,1) & = & n^{n-1} + \frac{1}{2} \sum_{i=1}^{n-1} \left( 
\binom{n-1}{i-1} i^{i-2} (n-i)! \sum_{\mu=1}^{n-i} 
\frac{(n-i)^{n-i-\mu}}{(n-i-\mu)!} \right. \\
&&+ \left. \binom{n-1}{i} (n-i)^{n-i-2} i! \sum_{\mu=1}^i 
\frac{i^{i-\mu}}{(i-\mu)!} \right) \\
&= &n^{n-1} + \frac{1}{2n} \sum_{i=1}^{n-1} \binom{n}{i} \left( i^{i-1}(n-i)! 
\sum_{\mu=1}^{n-i} \frac{(n-i)^{n-i-\mu}}{(n-i-\mu)!} \right. \\
& & + \left. (n-i)^{n-i-1}i!\sum_{\mu=1}^i \frac{i^{i-\mu}}{(i-\mu)!} \right) \\
&=& n^{n-1} + \frac{1}{n} \sum_{\mu=1}^{n-1} \sum_{i=1}^{n-\mu} \binom{n}{i} 
i^{i-1} (n-i)! \frac{(n-i)^{n-i-\mu}}{(n-i-\mu)!} \\
&=& n^{n-1} + \frac{1}{n}n! \sum_{\mu=1}^{n-1} \sum_{i=1}^{n-\mu} 
\frac{i^{i-1}}{i!} \frac{(n-i)^{n-i-\mu}}{(n-i-\mu)!}\,.
\end{eqnarray*}

Now, formula (\ref{AbelFormula}) yields for $n=m+\mu$:
$$\sum_{i=1}^m \frac{i^{i-1}}{i!} \frac{(m+\mu-i)^{m-i}}{(m-i)!} = \frac{1}{m!} 
\sum_{i=1}^m \binom{m}{i} i^{i-1} (m+\mu-i)^{m-i} = \frac{m(m+\mu)^{m-1}}{m!} 
.$$

Substituting this into the previous line we obtain
\begin{eqnarray*} J(n,1) & = & n^{n-1} + \frac{1}{n} n! \sum_{\mu=1}^{n-1} 
\frac{(n-\mu) n^{n-\mu-1}}{(n-\mu)!} = n^{n-1} + n! \sum_{\mu=1}^{n-1} 
\frac{n^{n-\mu-2}}{(n-\mu-1)!} \\
 & =&  n! \sum_{\mu=0}^{n-1} \frac{n^{n-\mu-2}}{(n-\mu-1)!} = n! \sum_{\mu=1}^n 
\frac{n^{n-\mu-1}}{(n-\mu)!}
\end{eqnarray*}

\noindent as desired. \end{proof}

A natural question is whether the recurrence can be similarly simplified for 
$k>1$.  The related sequences were not in the OEIS prior to this investigation, 
and examination of the graph enumeration literature, both fundamental 
(\cite{HaPa73,bergeron,comtet}) and more recent, did not uncover likely 
candidates for related generating functions.  Attempts with Abel-type 
identities of greater generality did not yield results that we considered 
useful, but the quest does not seem theoretically implausible; the authors 
would be interested in observations that readers might be able to provide.

\appendix 

\section{Appendix: Graph Theoretic Basics}\label{sec:basics}

We give the mathematical definitions related to Feynman graphs 
(without external edges) which are
relevant for understanding the definition of $I(n,k)$. Note that these concepts 
are only included here to make the paper 
self-contained. For simplicity in the following we shall refer to Feynman 
graphs simply as {\it{graphs}}. 

\bigskip

Let  $A$ and $B$ denote   sets.  By  
$[A]^2$ we denote the set of all $\two$-element subsets of $A$.  
Also, by $\two^A$, we denote the power set of $A$, i.e., the set of all subsets
of $A$.  
Finally, we
recall  that  the symmetric difference of the sets $A$ and $B$ is given by
$A\triangle B\defeq(A\cup B)\backslash (A\cap B)$.

\medskip

Let  $V=\{v_i\}_{i\in\mathbb{N}}$ and
$K=\{\legs_a\}_{a\in\mathbb{N}}$ be finite sets with $V\neq\emptyset$.
Let  $E\subseteq[K]^\two$ with $\cup_{a\in E}a=K$. Also, let the elements of $E$
satisfy
$\{\legs_a,\legs_{a'}\}\cap\{\legs_b,\legs_{b'}\}=\emptyset$. In this context, a {\it{graph}} is a pair $\graph=(V\times K,E)$ 
together
with  the following  mappings: 
\begin{enumerate}[(a)]
\item $\psi:K\to V$;
\item$\map:E\rightarrow [V]^\two\cup V; \{\legs_a,\legs_{a'}\}\mapsto
\{\psi(e_a),\psi(e_{a'})\}.$
\end{enumerate}
The elements of $V$ and $E$ are called {\it{vertices}} and {\it{edges}},
respectively. The edges are unordered pairs of elements of $K$. The elements of 
these pairs are called {\it{ends}} of edges.
The edges with both ends assigned to the
same vertex are also called {\it{loops}}.
Two or more edges joining the same pair of distinct vertices,
are called {\it{multiple edges}}.  The {\it{degree}} of a vertex is the number
of ends of edges
assigned to the vertex. Clearly, a loop adds $\two$ to the degree of a vertex.
For instance, Figure  \ref{fig:figure} (a)  shows a loop, while Figure
\ref{fig:figure} (b) shows  a graph with multiple edges. 

\begin{figure}
\begin{center}
\includegraphics[width=7cm]{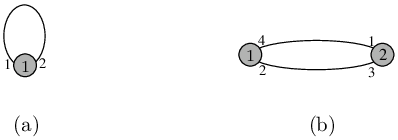}
\end{center}
\caption{(a) A loop; (b) A graph with multiple edges.}
\label{fig:figure}
\end{figure}

Here, a {\it{path}}  is a graph $P=(V\times K,E)$ together with the
mappings $\psi$ and $\map$, where $V=\{v_1,\ldots,v_{\vertex}\}$,
and
$\map(E)=\{\{v_1,v_\two\},\{v_\two, v_3\},\ldots,\{v_{n-1},v_\vertex\}\}$. 
Moreover, a {\it{cycle}}   is a  graph $C=(V'\times K',E')$ together 
with
the mappings $\psi'$ and $\map'$, where $V'=\{v_1,\ldots,v_{\vertex}\}$,
 and 
  $\map'(E')=\{\{v_1,v_\two\},\{v_\two,
v_3\},\ldots,\{v_{\vertex-1},v_\vertex\}, \{v_\vertex, v_1\}\}$. 
A  graph is said to be {\it{connected}} if every pair of vertices is 
joined
by a path. Otherwise, it is {\it{disconnected}}.
Moreover, a {\it{tree}}  is a connected graph with no cycles.

Furthermore, let  $\graph=(V\times K,E)$ together with the mappings $\psi$ and
$\map$ denote a graph. The set $\two^{E}$ is a vector space over the field
$\mathbb{Z}_\two$ so that vector addition is given by the  symmetric difference.
The {\it{cycle space}} $\CC$ of the graph $\graph$ is defined as the
subspace of $\two^{E}$ generated by all the cycles in $\graph$. The dimension of
$\CC$ is  called the {\it{cyclomatic number}} of the graph $\graph$. 

\pagebreak

Now, let $\graph=(V\times K,E)$ together with the mappings $\psi$ and $\map$,
and $\graph^*=(V^*\times K^*,E^*)$ together with the maps $\psi^*$ and $\map^*$
denote two graphs.  An {\it{isomorphism}} between the graphs $\graph$
and $\graph^*$ is a bijection  $f_{V\times K}\defeq f_V\times f_K:V\times K\to
V^*\times K^*$, where $f_V:V\to V^*$ and $f_K:K\to K^*$,  
which satisfies the following conditions:
\begin{enumerate}[(a)]
\item $\psi(e_a)=v_i\quad\mbox{iff}\quad\psi^*(\iso_K(e_a))=\iso_V(v_i),$
\item$\map(\{\legs_{a},\legs_{a'}\})=\{v_i,v_{i'}\}\quad \mbox{iff}\quad
\map^*(\{f_K(\legs_{a}),f_K(\legs_{a'})\})=\{f_V(v_i),f_V(v_{i'})\}.$
\end{enumerate}
Clearly, an isomorphism defines an equivalence relation on graphs. In this 
context, an automorphism of a graph $\graph$ is an isomorphism of the 
graph  onto itself. It is easy to verify that the order of the groups of 
automorphisms of the graph of Figure \ref{fig:figure} (a) is 2, while that 
of Figure \ref{fig:figure} (b) is 4. 

Finally, note that  in the context of Section 
\ref{sec:enumerator}, the number of connected graphs on the set 
$\{1,...,n\} \times 
\{1,...,2(n+k-1)\} $ is $(2(n+k-1))!n!I(n,k)$.

\bigskip

\label{lastpage-01}

\end{document}